\documentclass{article}

\usepackage{amsfonts}
\usepackage{amssymb}
\usepackage{amsmath}
\usepackage{amsthm}
\usepackage{color}
\usepackage{graphicx}

\theoremstyle{plain}
\newtheorem{theorem}{Theorem}[section]

\newtheorem{lemma}[theorem]{Lemma}

\theoremstyle{definition}

\newtheorem{example}[theorem]{Example}

\numberwithin{equation}{section}

\newcommand{\NN}{{\mathbb N}}
\newcommand{\ZZ}{{\mathbb Z}}

\newcommand{\QQ}{{\mathbb Q}}
\newcommand{\PP}{{\mathcal P}}
\newcommand{\BB}{{\mathcal B}}

\title{Pseudoquotient extensions of measure spaces}
\author{Piotr Mikusi\'nski\\
\small{University of Central Florida}}

\begin{document}

\maketitle

\begin{abstract}  A space of pseudoquotients $\PP(X,S)$ is defined as equivalence classes of pairs $(x,f)$, where $x$ is an element of a non-empty set $X$, $f$ is an element of $S$, a commutative semigroup of injective maps from $X$ to $X$, and $(x,f) \sim (y,g)$ if $gx=fy$. In this note we assume that $(X,\Sigma,\mu)$ is a measure space and that $S$ is a commutative semigroup of measurable injections acting on $X$ and investigate under what conditions there is an extension of $\mu$ to $\PP(X,S)$.    
\end{abstract}

\section{Introduction}

The construction of pseudoqutients was introduced in \cite{Quo} under the name of ``generalized quotients." A space of pseudoquotients, denoted by $\PP(X,S)$, is defined as the space of equivalence classes of pairs $(x,f)\in X \times S$, where $X$ is a non-empty set and $S$ is a commutative semigroup of injective maps from $X$ to $X$.  The equivalence relation is defined as follows: $(x,f) \sim (y,g)$ if $gx=fy$. It is a generalization of the constructions of the field of quotients from an integral domain. For example, if we take $X=\ZZ$ and $S=\NN$ acting on $\ZZ$ by multiplication, then  $\PP(\ZZ,\NN)=\QQ$.

Pseudoquotients have desirable properties. The set $X$ can be identified with a subset of $\PP(X,S)$ and the semigroup $S$ can be extended to a commutative group of bijections acting on $\PP(X,S)$ (see \cite{KimPM} or \cite{Bolu}).  Under natural conditions on $S$, the algebraic structure of $X$ extends to $\PP(X,S)$. Pseudoquotient extensions also have good topological properties (see \cite{BMR}, \cite{CBMR}, \cite{AKPM1}, \cite{KimPM}, and \cite{MP}). A version of the pseudoquotient extension for noncommutative semigroup $S$ is presented in \cite{AKPM2}.

Pseudoquotients can be a useful tool in analysis (see, \cite{Levy}, \cite{Radon}, \cite{AMN}, \cite{Quo}, or \cite{Bolu}).

\section{Pseudoquotients}

Let $X$ be a nonempty set and let $S$ be a commutative semigroup of injective maps from $X$ into $X$.  For $(x,f), (y,g) \in  X \times S$ we write $(x,f) \sim (y,g)$  if $g x=f y$. It is easy to check that this is an equivalence relation in $X \times S$.  We define $\PP(X,S)  = (X \times S)/\sim$.  If there is no ambiguity, we will write $\PP$ instead of $\PP(X,S)$. Elements of $\PP$ will be called pseudoquotients.  The equivalence  class of $(x,f)$ will be denoted by $\frac{x}{f}$. This is a slight abuse of notion, but we follow here the tradition of denoting rational numbers by $\frac{p}{q}$ even though the same formal problem is present there.

Elements of $X$ can be identified with elements of $\PP $ via the embedding $\iota: X \to \PP(X, S)$ defined by
\[
\iota (x)= \frac{f x}{f}, 
\]
where $f$  is an arbitrary element of $S$. Action of $S$ can be extended to $\PP $ via
\[
f \frac{x}{g} = \frac{f x}{g}.
\]
If, for some $y \in X$, $f \frac{x}{g} = \iota(y)$, then we will write $f \frac{x}{g} \in X$ and $f \frac{x}{g}  = y$, which is formally incorrect, but convenient and harmless. For instance, we have  the natural equality $f \frac{x}{f} = x$.

Elements of $S$, when extended to maps on $\PP $, become bijections.  Indeed, for any $f\in S$ and $\frac{x}{g}\in \PP $, we have
\[
f \frac{x}{fg} = \frac{f x}{fg}= \frac{x}{g}.
\]
The action of $f^{-1}$ on $\PP $ can be defined as
\[
f^{-1}\frac{x}{g} = \frac{x}{fg}
\]
and $S$ can be extended to a commutative group 
$
G=\{f^{-1}g : f,g\in S\}
$
of bijections acting on $\PP $ .

If $f:X\to X$ is an injective map, then $S=\{f^n: n=0,1,2,\dots\}$ is a commutative semigroup of injections. We will refer to such an $S$ as the semigroup generated by $f$.

\section{Pseudoquotients on measurable spaces}

Let $(X,\Sigma)$ be a measurable space and let $S$ be a commutative semigroup of measurable injections acting on $X$ such that $f(A)\in\Sigma$ for all $f\in S$ and all $A\in\Sigma$. By $\Sigma_\PP$ we mean the $\sigma$-algebra in $\PP$ generated by the set
\[
\BB=\{A\subset \PP: f(A)\subset X \text{ and } f(A)\in \Sigma \text{ for some } f\in S\}.
\]

In what follows, when we write $f(A)\in \Sigma$ we will always mean that $f(A)\subset X$  and $f(A)\in \Sigma$.

\begin{lemma}
 Let $A,B\in\BB$ and let $A\subset B$. If $f(B)\in \Sigma$ for some $f\in S$, then $f(A)\in \Sigma$.  
\end{lemma}

\begin{proof}
If $g(A)\in\Sigma$ for some $g\in S$, then
 \[
 f(A)=f(g^{-1}(g(A)))=g^{-1}(f(g(A)).
 \]
 Since $g(A)\in\Sigma$, $f(g(A))\in\Sigma$, and consequently $g^{-1}(f(g(A))\in\Sigma$, because $f(A)\subset f(B)\subset X$. 
\end{proof}

\begin{theorem}
 Elements of the group $G$ of bijections on $\PP$ are measurable with respect to the $\sigma$-algebra $\BB$.
\end{theorem}

\begin{proof}
 It suffices to show that $fg^{-1}(A) \in \BB$ for every $f,g\in S$ and $A\in\BB$. If $A\in\BB$, then $h(A)\in \Sigma$ for some $h\in S$. Since $gh\in S$ and
 \[
(gh)fg^{-1}(A)=fh(A)\in\Sigma,
 \]
 we have $fg^{-1}(A)\in\BB$.
 \end{proof}

\section{Extension of measures}

A measurable map $f:(X,\Sigma,\mu)\to(X,\Sigma,\mu)$ is called $\mu$-{\it homogeneous}, or just {\it homogeneous}, if there is a constant $\alpha>0$ such that $\mu(f(A))=\alpha\mu(A)$ for every $A\in\Sigma$.  

\begin{theorem} Let $(X,\Sigma,\mu)$ be a measure space and let $S$ be a commutative semigroup of homogeneous measurable injections acting on $X$. Then there is a unique extension of $\mu$ to $\mathcal{P}$. 
\end{theorem}

\begin{proof} Let $A,B\in\BB$. Then there are $f,g\in S$ such that $f(A)\in\Sigma$ and $g(B)\in\Sigma$. Then $f(g(A)),f(g(B))\in\Sigma$ and we have
\[
f(g(A\cup B))=f(g(A)) \cup f(g(B))\in\Sigma
\]
and
\[f(g(A\setminus B))=f(g(A)) \setminus f(g(B))\in\Sigma.
\]
Thus, if $A,B\in\BB$, then $A\cup B, A\setminus B\in \BB$. In other words, $\BB$ is a ring of subsets of $\mathcal{P}$. 

Now let $B\subset \mathcal{P}$ be such that $f(B)\in\Sigma$ for some $f\in G$. If $\mu(f(A))=\alpha\mu(A)$ for every $A\in\Sigma$, then we define
\[
\mu(B)=\frac1{\alpha}\mu(f(B)).
\]

To show that $\mu(B)$ is well-defined suppose  $g(B)\subset X$ for some $g\in G$ and $\mu(g(A))=\beta\mu(A)$ for every $A\in\Sigma$. Since
\[
f^{-1}(f(B))=g^{-1}(g(B))
\]
we have
\[
g(f(B))=f(g(B))
\]
and hence
\[
\beta\mu(f(B))=\mu(g(f(B)))=\mu(f(g(B)))=\alpha \mu(g(B)).
\]

Now assume that $A_1, A_2, \dots \in \BB$ are pairwise disjoint and $\cup_{n=1}^\infty A_n \in \BB$. We need to show that 
\[
\mu\left( \bigcup_{n=1}^\infty A_n \right) = \sum_{n=1}^\infty \mu (A_n).
\]

Since $\cup_{n=1}^\infty A_n \in \BB$, then there is $f\in S$ such that
\[
f\left( \bigcup_{n=1}^\infty A_n\right)= \bigcup_{n=1}^\infty f(A_n) \in \BB.
\]
Note that for every $n\in\NN$ we have $f(A_n) \subset f\left( \cup_{n=1}^\infty A_n\right) \in \Sigma$, so $f(A_n)\in \Sigma$. If $\mu(f(A))=\alpha\mu(A)$ for every $A\in\Sigma$, then
\begin{align*}
\mu\left( \bigcup_{n=1}^\infty A_n \right) 
&= \frac1{\alpha} \mu \left(f\left( \bigcup_{n=1}^\infty A_n \right) \right)
= \frac1{\alpha} \mu \left(\bigcup_{n=1}^\infty f(A_n) \right)\\
&= \sum_{n=1}^\infty  \frac1{\alpha} \mu(f(A_n))
= \sum_{n=1}^\infty \mu (A_n).
\end{align*}
Consequently, $\mu$ has a unique extension to a measure on $\mathcal{P}$ with the $\sigma$-algebra generated by $\BB$.
\end{proof}

\section{Examples}

\begin{example}
Let $X=[-1,1]$ and let $\mu$ be the Lebesgue measure on the $\sigma$-algebra of Borel subsets of $[-1,1]$. Let $S$ be the semigroup generated by the function $f(x)=x/2$. Then $S$ is a commutative semigroup acting on $[-1,1]$ injectively. 

In this case  $\mathcal{P}(X,G)$ can be identified with $\mathbb{R}$ and $\Sigma_\PP$ is the $\sigma$-algebra of Borel subsets of $\mathbb{R}$. Since for any measurable $A\subset[-1,1]$ and any $n=0,1,2,\dots$ we have
$$
\mu(f^n(A))=\frac1{2^n} \mu(A),
$$
$\mu$ has a unique extension to a measure on $\mathbb{R}$, which is the Lebesgue measure on $\mathbb{R}$.
\end{example}

Note that in the above example the $\sigma$-algebra $\Sigma_\PP$ restricted to $[-1,1]$ is the original $\sigma$-algebra. It is not always the case, as the second example shows. 

\begin{example}
Let $X=\mathbb{R}$ with the $\sigma$-algebra $\Sigma$ generated by all sets of the form $[k,k+1)$ where $k\in\mathbb{Z}$. We define $\mu(\left[k,k+1 \right))=1$ and extend $\mu$ to $\Sigma$ by $\sigma$-additivity. Let $S$ be the semigroup generated by the function $f(x)=2x$. Then $S$ is a commutative semigroup acting on $\mathbb{R}$ injectively. Moreover, $f^n$ maps sets from $\Sigma$ to sets in $\Sigma$ and we have
$$
\mu(f^n(A))=2^n \mu(A).
$$
In this case  $\mathcal{P}(X,G)=X=\mathbb{R}$ and $\Sigma_\PP$ is the $\sigma$-algebra of Borel subsets of $\mathbb{R}$. The unique extension of $\mu$ to a measure on $\mathbb{R}$ is the Lebesgue measure.
\end{example}

In the above example the pseudoquotient extension does not change the set $X$ but extends a very simple $\sigma$-algebra on $\mathbb{R}$ with a rather primitive measure to the $\sigma$-algebra of Borel subsets of $\mathbb{R}$ with the Lebesgue measure.

\vspace{2mm}

{\bf Acknowledgment.} I would like to thank Jason Bentley for his helpful comments.

\end{document}